\documentclass[12pt,reqno]{amsart}
\usepackage{amsaddr}
\usepackage{cite}
\usepackage[colorlinks,citecolor=blue,urlcolor=blue,hypertexnames=false, pagebackref]{hyperref}
\usepackage[capitalize,nameinlink]{cleveref}

\usepackage{amssymb,amsmath,amsfonts,latexsym}
\usepackage{ragged2e}
\usepackage{bm}
\usepackage{mathtools}
\usepackage{enumerate}
\usepackage{array, graphics,xcolor}
\allowdisplaybreaks
\usepackage{mathrsfs}
\usepackage{mathtools}
\usepackage{blkarray, bigstrut}
\usepackage{color, bm, amscd, tikz-cd, mathtools}
\usepackage{csquotes}
\newcommand{\bydef}{\stackrel{\rm def}{=}}

\usepackage{tikz}

\newtheorem{thm}{Theorem}[section]

\newtheorem{cor}[thm]{Corollary}
\newtheorem{lem}[thm]{Lemma}

\newtheorem{theorem}{Theorem}
\newtheorem{defn}[thm]{Definition}
\newtheorem{rem}[thm]{Remark}

\numberwithin{equation}{section}

\def\tilde{\widetilde}

\def\C{\mathbb{C}}

\def\diag{\mathrm{diag}}

\def\epsilon{\varepsilon}

\def\bH{\mathbb{H}}
\def\H{\mathcal{H}}

\def\K{\mathcal{K}}

\def\N{\mathbb{N}}
\def\nn{\nonumber}

\def\phi{\varphi}

\def\T{\underline{T}}

\def\X{\underline{X}}
\def\Y{\underline{Y}}
\def\W{\underline{W}}
\def\Z{\underline{Z}}
\def\bM{\mathbb{M}}
\def\cM{\mathcal{M}}
\def\bU{\mathbb{U}}

\author{Tirthankar Bhattacharyya}
\address{Department of Mathematics, Indian Institute of Science\\
	Bangalore 560012, India}
\email{tirtha@iisc.ac.in}

\author{James Eldred Pascoe}
\address{Department of Mathematics, Drexel
	University\\ Philadelphia, Pennsylvania,
	USA}
\email{jep362@drexel.edu}

\author{Chandan Pradhan}
\address{Department of Mathematics and Statistics, 	University of New Mexico\\
	Albuquerque, NM 87106, USA}
\email{chandan.pradhan2108@gmail.com}

\usepackage{fancyhdr}
\begin{document}
		
	\title[Nc-Carath\'eodory approximation]{Finite sum of squares, finite realization and noncommutative Carath\'eodory approximation}
	
	\maketitle

\begin{abstract}
In the noncommutative polydisc, we first prove a positive sum of squares formula for a non-negative hereditary rational nc-function. The number of summands is finite. This result is used to derive a finite-dimensional realization formula for contractive nc-rational functions, where the colligation matrix is contractive. It is unitary if and only if the function is inner. Finally, we apply these results to generalize Carathéodory’s classical theorem - approximating holomorphic self-maps of the unit disc by finite Blaschke products - to the setting of holomorphic functions on the noncommutative polydisc. This is in sharp contrast with the commutative situation where Carathéodory’s approximation is known for Schur classes only in the unit disc and the unit bidisc.

\end{abstract}

\section{Introduction}

\subsection{Finite sum of squares and finite dimensional realization}
 Let $n, d\in\N$. Let $\bM_n$ denote the 
space all $n\times n$ complex matrices over complex numbers with the operator norm, and let $\bM_n^d$ denote $d$ copies of $\bM_n$, that is  
$\bM_n^d=\bM_n\times\cdots \times \bM_n$ ($d$ times). Set \[\bM^d:=\bigsqcup_{n=1}^{\infty} \bM_n^d.\]
Let $\bU_n$ denotes the set of all $n\times n$ unitary matrices, and set $\bU_n^d=\bU_n\times\cdots\times\bU_n^d$ ($d$ times). For brevity, let $\X$ denote the tuple $(X_1, \ldots, X_d)$ and $\|\X\|=\max_{1\leq i \leq d}\|X_i\|$.
Let 
\begin{align*}
	\mathbb{D}^d_{\rm nc} = \{\X \in \bM^d : \|\X\| < 1\}, \quad \text{and} \quad	\partial\mathbb{D}^d_{\rm nc} = \{\X \in \bM^d : X_i^*X_i = I, i=1,\ldots,d\}.
\end{align*}
This $\mathbb{D}^d_{\rm nc}$ is called the {\em noncommutative polydisc (nc-polydisc)}.

\begin{defn}
	A {\em noncommutative polynomial (an nc-polynomial), also called a free polynomial}, in variables $(X_1,\ldots,X_d)$ on $\mathbb{M}^d$ is a finite linear 
combination, with complex coefficients, of words formed from the 
	non-commuting symbols
	\[
	X_1,\ldots,X_d
	\]
with the empty word denoting the constant function $1$. It is an element of the free associative algebra over $\mathbb C$ in $d$ (non-commuting) 
variables.	A $*$-nc polynomial in variables $(X_1,\ldots,X_d)$ on $\mathbb{M}^d$ is a finite linear combination, with complex coefficients, of words formed from the 
	non-commuting symbols
	\[
	X_1,\ldots,X_d, X_1^*,\ldots,X_d^*.
	\]
\end{defn}
For example, $X_1X_2-X_2X_1$ is an nc polynomial and $X_1X_2^*-X_2X_1$ is a $*$-nc polynomial.


\begin{defn}
	A {\em noncommutative rational expression (an nc-rational expression)}  in variables $(X_1,\ldots,X_d)$ on $\mathbb{M}^d$ is an expression in terms of 
addition, multiplication, scalar multiplication 
	and inverses of the noncommuting symbols $	X_1,\ldots,X_d.$  A $*$-nc rational expression in variables $(X_1,\ldots,X_d)$ on $\mathbb{M}^d$ is an 
expression in terms of addition, multiplication, scalar multiplication and inverses of $X_1,\ldots,X_d, X_1^*,\ldots,X_d^*$.
\end{defn}
For example $(X_1X_2-X_2X_1)^{-1}X_3, 1$ and $X_1X_1^{-1}$ are nc rational expressions. The latter two are equal where they are defined. 
$X_2^*(X_1X_2-X_2X_1)^{-1}X_3, (X_1^*)^{-1}X_2$ are examples of $*$-nc rational expressions.
\begin{defn}
	A {\em $*$-noncommutative rational function ($*$-nc rational function)} in variables $(X_1,\ldots,X_d)$ on $\mathbb{M}^d$ is an equivalence class of 
$*$-nc rational expressions, where two $*$-nc rational 
	expressions with nonempty domains are equal if they are the same on the intersection of their domains of definition. The definition of an nc analytic 
rational function is similar.
\end{defn}
For example $X_1X_1^{-1}$ equals $1$ as a rational function. 

An nc-rational function $\Theta$ has limiting boundary values because $\bM_n^d \cap \partial\mathbb{D}^d_{\rm nc}$ is $\bU_n^d$ and for a fixed 
$d$-tuple $(U_1, U_2, \ldots ,U_d)$ of $n \times n$ unitary matrices, the function $\Theta(zU_1, zU_2, \ldots zU_d)$ is a matrix valued rational function on 
the open unit disc $\mathbb D$ which has radial limits at all points of the unit circle (a consequence of the disc version of the classical 
Kalman-Yakubovich-Popov lemma), i.e., $\lim_{r \uparrow 1} \Theta(rU_1, rU_2, \ldots rU_d)$ exists. An nc-rational function $\Theta$ defined on $\mathbb{D}^d_{\rm nc}$ is 
called {\em inner} if $f(\X) \in \bU_n$ for almost all $\X \in \bM_n^d \cap \partial\mathbb{D}^d_{\rm nc}$ with respect to the natural Haar measure of the group $\bU_n^d$.

Artin's famous solution to Hilbert's 17th problem states that for $d \ge 2$, a real polynomial $p$ in $d$ variables is non-negative on $\mathbb R^d$ if and only if there is a finite number of rational functions $q_j$ satisfying $p - \sum_{\text{finite}} q_j^2$. In a remarkable result which has initiated a lot of research, it was proved in \cite{He02} that in the noncommutative case rational functions are not needed, i.e., a $*$-nc polynomial \emph{with the matrices having real entries and the coefficients being real} is non-negative if and only if $p = \sum_{\text{finite}} q_j^* q_j$. The complex case was resolved in \cite{McPu05}. This was followed by \cite{Pa18} and \cite{Vo21} where rational functions were dealt with. 

\begin{defn}
	A $*$-nc rational function $q$ is said to be \emph{hereditary} if 
	\[
	q \equiv \sum_{\text{finite}} r_i^* s_i,
	\]
	where each $r_i$ and $s_i$ is an nc rational function.
\end{defn}

Our first major result extends the positive sums of squares results known so far. This result was tacitly implied in \cite{Pa18}. We find that the approach there does not directly yield a complete proof. We therefore develop a different method that allows us to establish the result rigorously.

\begin{theorem} \label{Main1}
	Let $q$ be a hereditary rational nc function.
	If $q \geq 0$ on $\mathbb{D}^d_{\rm nc}$, then 
	\begin{align}
		q \equiv \sum_{\text{finite}} r_i^* (1-X_{i_k}^*X_{i_k}) r_i 
		+ \sum_{\text{finite}} s_i^* s_i
	\end{align}
for nc rational functions $r_i$ and $s_i$.
\end{theorem}

This sum of squares formula leads to a realization of contractive nc-rational (inner) functions on $\mathbb{D}^d_{\rm nc}$ in terms of finite realizations. In the commutative case, this realization was proved in \cite{Kummert}. See \cite{Kn21} for a recent exposition. We note that for a rational formal power series (FPS) (in the language of \cite{AlKa06}) $f$ on $\mathbb{D}_{\mathrm{nc}}^d$ satisfying \eqref{eq:contractive}, it is shown in \cite{AlKa06} that $f$ admits a finite realization for which the colligation matrix is unitary if $f$ is inner. See also \cite[Theorem 2.1]{BaGrMa05}. In the present work, we provide a proof based on our Theorem \ref{Main1} and the techniques developed herein which are very different from those in \cite{AlKa06, BaGrMa05}. We obtain finite dimensional realizations via a prior constraint on the vector space of functions, as opposed to a controllability-observability argument as a tool to reduce to finite dimensions.

\begin{theorem}\label{Main2}
	 Suppose $f$ is an nc rational function such that 
	\begin{align}\label{eq:contractive}
		1 - f(\Z)^* f(\Z) \geq 0 \quad \text{on } \mathbb{D}^d_{\rm nc}.
	\end{align}
	Then there exist a finite-dimensional Hilbert space $\cM$ and a contraction 
	\[
	T=\begin{bmatrix}
		A & B \\[3pt]
		C & D
	\end{bmatrix} :
	\begin{matrix}
		\C \\ \oplus \\ \cM
	\end{matrix}
	\to
	\begin{matrix}
		\C \\ \oplus \\ \cM
	\end{matrix}
	\]
	such that 
	\begin{align}\label{eq:realization-form}
		f(\Z) = A + B\, \Delta(\Z)\, (1 - D\, \Delta(\Z))^{-1} C
	\end{align}
	on $\mathbb{D}^d_{\rm nc}$, where $\Delta(\Z)$ is a block diagonal $\dim\cM \times \dim\cM$ matrix with the diagonal entries coming from $(Z_{1},\cdots, Z_{d})$.  More explicitly, for every $\Z \in \mathbb{D}^d_{\rm nc}\cap\bM_n^d$,
	\begin{align*}
		\begin{matrix}
			I_{\C} \\ \otimes \\ f(\Z)
		\end{matrix}=\begin{matrix}
			A \\ \otimes \\ I_{\C^n}
		\end{matrix} + 
		\begin{matrix}
			B \\ \otimes \\ I_{\C^n}
		\end{matrix}
		\Delta(\Z) \left(\begin{matrix}
			I_{\cM} \\ \otimes \\ I_{\C^n}
		\end{matrix} -   \begin{matrix}
			D \\ \otimes \\ I_{\C^n}
		\end{matrix}
		\Delta(\Z)\right)^{-1} \begin{matrix}
			C \\ \otimes \\ I_{\C^n}
		\end{matrix}.
	\end{align*}
The formula \eqref{eq:realization-form} is referred to as the \emph{realization formula} for $f$, where $T$ denotes the corresponding \emph{colligation 
matrix}.


\end{theorem}
\begin{cor}\label{thm:realization-inner-rational}
A nc rational function $f$ on $\mathbb{D}^d_{\rm nc}$ is \emph{inner} if and only if there exist a finite unitary colligation matrix $T$ satisfying \eqref{eq:realization-form}.
\end{cor}

\subsection{Carath\'eodory's approximation}
For $\Omega = \mathbb D$ or $\Omega = \mathbb D^2$, a holomorphic matrix valued function $F$ on $\Omega$ is called $rational$ if every entry is a rational 
function with the poles off $\Omega$ and is called $inner$ if the boundary values of the function on the unit circle/torus are unitary matrices almost 
everywhere.
	
	Carath\'eodory, in his study of holomorphic functions from the open unit disc $\mathbb D = \{ z \in \mathbb C : |z| < 1 \}$ of the complex plane to the 
closed unit disc $\overline{\mathbb D}$, proved the following theorem, see Section 284 in \cite{Ca54}. Later, Rudin generalized this to functions defined on the 
polydisc $\mathbb D^n$, see \cite{Ru69}. We shall state it for functions taking values in matrices. 

\vspace*{2mm}

\noindent	\textbf{Theorem CR.} \quad (Carath\'eodory and Rudin)  \quad
		{\em Let $\Omega$ denote the open unit disc $\mathbb D$ or the bidisc $\mathbb D^2$. Any holomorphic function $\varphi$ defined on $\Omega$, taking 
values in square matrices (of fixed size) and having norm $\| \varphi \| (\bydef \sup \{ \| \varphi(z) \| : z \in \Omega \})$ no larger than $1$ can be 
approximated (uniformly on compact subsets) by rational inner functions.}

\vspace*{2mm}

We stated it only for the disc or the bidisc for two reasons. First, in the polydisc, one needs to restrict to the so-called Schur-Agler class. This matrix-valued version of Rudin’s result is not
known if $\varphi$ is in Schur class. Secondly, an application of unitary dilation of a contraction proves Theorem CR, see \cite{AlBhJiKu23}. We are going to follow this route of application of Hilbert space operator theory. Moreover, in the noncommutative case this technique will work in the noncommutative (nc) {\em polydisc}. A major tool is the finite (dimensional) realization obtained above. Note that the nc-polydisc is closed under direct sum and unitary conjugation. 
\begin{defn}
A {\em free function} or an {\em nc-function} $f$ on $\mathbb{D}^d_{\rm nc}$ is an $\mathbb M^1$ valued function which is 
\begin{enumerate}
\item graded, that is, if $\X \in \bM_n^d \cap \mathbb{D}^d_{\rm nc}$, then $f(\X) \in \bM_n$ and
\item satisfies the intertwining property, that is, if $\X$ and $\Y$ are in $\mathbb{D}^d_{\rm nc}$ and $s$ is any (possibly rectangular) matrix such that $s\X = \Y s$, 
    then $s f(\X) = f(\Y) s$.
    \end{enumerate} 
 
 We say that an nc-function $f$ on $\mathbb{D}^d_{\rm nc}$ is {\em holomorphic} on $\mathbb{D}^d_{\rm nc}$ if for each $n \ge 1$, the map defined on $\mathbb{D}^d_{\rm nc} 
 \cap \mathbb{M}^d_n$ by $x \mapsto f(x)$ is a holomorphic $\mathbb{M}_n^1$-valued function in the entries of $x$.
 \end{defn}
 It is a remarkable fact that if an 
 nc-function $f$ satisfies $\|f(\Z)\| \le 1$ for every $\Z \in \mathbb{D}^d_{\rm nc}$, then it is automatically nc-holomorphic on $\mathbb{D}^d_{\rm nc}$. This follows from Theorem 4.6 of \cite{AgMc15}. We are ready to state the following noncommutative version of Theorem CR.

\begin{theorem} \label{Main}
Let $f$ be an nc-holomorphic function on $\mathbb{D}^d_{\rm nc}$ satisfying $\|f(\Z)\| \le 1$ for every $\Z \in \mathbb{D}^d_{\rm nc}$. Then there exists a sequence $\{\Theta_n\}$ of 
nc-rational inner functions on $\mathbb{D}^d_{\rm nc}$ such that $f$ can be approximated by $\Theta_n$ uniformly on each compact (in du-topology, that is for each 
$n\in\N$, $K\cap \bM_n^d$ is compact in usual topology on $\bM_n^d$) subset $K$ of $\mathbb{D}^d_{\rm nc}$.
	
\end{theorem}

\section{Finite sums of squares and realization}
	
\begin{defn}
    A descriptor realization for an nc rational function $r$ is
       $$r(\X)=b^*\Lambda(\X)^{-1}c$$
    where $\Lambda$ is a matrix valued affine linear function of $X_1,\ldots, X_d, X_1^*, \ldots, X_d^*.$
\end{defn}
For example
    $$-X^{-2}= \begin{pmatrix} 1 \\ 0 \end{pmatrix}^*\begin{pmatrix} 0 & X \\ X & 1 \end{pmatrix}^{-1}\begin{pmatrix} 1 \\ 0 \end{pmatrix}.$$
Given a descriptor realization for two functions $f, g$ we may easily find realizations for $fg, f+g, f^{-1},$ (see e.g. \cite{Vo18, KaVi14}).
Thus, every rational function possesses a descriptor realization. If $r$ is defined on a neighborhood of $0,$ we may take $\Lambda (0)=1$ and shall adopt this 
convention henceforth.

By applying realizations term by term we see the following.
\begin{rem} \label{Lambda}
    Every hereditary $*$-nc rational function has a realization of the form:
        $$q(\X)=(b \otimes 1)^*\Lambda(\X)^{-1*}M\Lambda(\X)^{-1}(b \otimes 1)$$
    where $b \in \mathbb C^N$ and $\Lambda$ is an $N \times N$ matrix-valued affine analytic linear function for some $N$ and $M$ is some $N \times N$ matrix.
\end{rem}

Lemma 3.9 of \cite{Vo18} shows that the entries of $\Lambda^{-1}$ are nc rational functions. We shall now work towards a Positivstellensatz.

\begin{defn}
    
    Let $\Lambda$ be a square matrix of affine linear analytic nc polynomials. 	Define the principal analytic functions associated to $\Lambda$
    to be the set of nc analytic rational functions spanned by the entries of $\Lambda^{-1}$ and the constant polynomial $1.$ Call this vector space $\mathcal{V}_p$.

    The secondary analytic functions are the span of the principal analytic functions and $X_ih$ where $h$ is a principal analytic function. Call this vector space $\mathcal{V}_s$. 
    
    The principal hereditary functions are the span of $h^*k$ where $h, k$ principal analytic functions.

    The secondary hereditary functions are the span of $h^*k$ where $h, k$ secondary analytic functions. Call this vector space $\mathcal{V}$.

    The principal cone consists of elements of the form
        $h^*(1-X_i^*X_i)h$ where $h$ are principal analytic functions and $h^*h$ where $h$ are  secondary analytic functions. Call this cone $\mathcal{C}.$
\end{defn}
For example, if $\Lambda = \begin{pmatrix} 1 & X \\ 0 & 1 \end{pmatrix}$, then $\Lambda^{-1} = \begin{pmatrix} 1 & -X \\ 0 & 1 \end{pmatrix}.$
Thus the primary nc analytic rational functions in this case are spanned by $1, X,$ the secondary by $1, X, X^2.$

The principal cone sits inside the finite dimensional vector space $\mathcal V$ of all secondary hereditary functions on which there is a canonical topology of uniform convergence on closed nc sub-polydiscs.

\begin{lem}
	The principal cone $\mathcal{C}$ is closed in $\mathcal{V}$.
\end{lem}

\begin{proof}
	Let $\dim \mathcal{V}_p = N$ and $\dim \mathcal{V}_s = M$, and fix bases
	\[
	\mathcal{V}_p = \mathrm{span}\{p_1,\dots,p_N\}, 
	\qquad 
	\mathcal{V}_s = \mathrm{span}\{s_1,\dots,s_M\}.
	\]
	For $h = \sum_{i=1}^M \beta_i s_i \in \mathcal{V}_s$, we have
	\[
	h^* h 
	= \sum_{i,j=1}^M \overline{\beta_i}\beta_j \, s_i^* s_j.
	\]
	Similarly, for $h = \sum_{k=1}^N \alpha_k p_k \in \mathcal{V}_p$,
	\[
	h^*(1 - X_i^* X_i)h
	= \sum_{k,\ell=1}^N \overline{\alpha_k}\alpha_\ell \, p_k^*(1 - X_i^* X_i)p_\ell.
	\]
	Let $\bM_N^+$ denote the cone of $N \times N$ positive semidefinite matrices. Then $(\overline{\alpha_k}\alpha_\ell) \in \bM_N^+$ and $(\overline{\beta_i}\beta_j) \in \bM_M^+$. Hence every element of $\mathcal{C}$ admits a representation of the form
	\begin{equation}\label{GeneralForm}
		\sum_{i=1}^d \sum_{k,\ell=1}^N A^{(i)}_{k\ell}\, p_k^*(1 - X_i^* X_i)p_\ell
		\;+\;
		\sum_{r,s=1}^M B_{rs}\, s_r^* s_s,
	\end{equation}
	with $A^{(i)} \in \bM_N^+$ and $B \in \bM_M^+$.
	
	Define the linear map
	\[
	\Phi : (\mathbb{M}_N)^d \times \mathbb{M}_M \to \mathcal{V}
	\]
	by sending $(A_1,\dots,A_d,B)$ to the expression in \eqref{GeneralForm}. Let $	D $ denote the closed convex cone $ (\bM_N^+ \times \bM_M^+) $  in $(\mathbb{M}_N)^d \times \mathbb{M}_M$. 	Consider the slice
	\[
	D_1 := \left\{ (A_1,\dots,A_d,B)\in D :
	\sum_{i=1}^d \mathrm{Tr}(A_i) + \mathrm{Tr}(B) = 1 \right\}.
	\]
	Since $D$ is closed and the trace is continuous, $D_1$ is compact. Set $K := \Phi(D_1) \subset \mathcal{V}$. By continuity of $\Phi$, the set $K$ is compact. 	
	For $(A_1,\dots,A_d,B)\in D$, define
	\[
	t := \sum_{i=1}^d \mathrm{Tr}(A_i) + \mathrm{Tr}(B) \ge 0.
	\]
	If $t>0$, then
	\[
	(A_1,\dots,A_d,B) = t(\tilde A_1,\dots,\tilde A_d,\tilde B),
	\quad (\tilde A_1,\dots,\tilde A_d,\tilde B)\in D_1,
	\]
	and hence
	\[
	\Phi(A_1,\dots,A_d,B) = t\,\Phi(\tilde A_1,\dots,\tilde A_d,\tilde B).
	\]
	It follows that
	\[
	\mathcal{C} = \{ t y : t \ge 0,\ y \in K \} = \mathrm{cone}(K).
	\]
	Since the cone generated by a compact set in a finite-dimensional vector space is closed, we conclude that $\mathcal{C}$ is closed.
\end{proof}

\begin{lem}
    Let $\epsilon > 0$. Let a $*$-nc hereditary rational function $q$ be defined on all of $\{\X \in \bM^d : \|\X\| < 1 + \epsilon \}$.  Then $q$  is 
    in the principal cone  if and only if $q(\X)$ is positive semidefinite for any $\X$ in $\{\X \in \bM^d : \|\X\| < 1\}$.
\end{lem}
\begin{proof}
Suppose $q$ is not in the principal cone. We shall prove that there is a $\Y$ such that $q(\Y)$ fails to be positive semidefinite. 

By the Hahn-Banach cone separation theorem, there is a continuous linear functional $\tau$ on the space of secondary hereditary functions such that 
$\tau(q)$ is negative and $\tau$ is  non-negative on the principal cone.

Define two Hilbert spaces $\mathcal H_1$ and $\mathcal H_2$ corresponding to the principal analytic functions and the secondary analytic functions 
respectively.
(That is, define $\tau(f^*g)= \langle g,f \rangle.$ As $f^*f$ is in the cone, we have a pre-inner product. Quotient out the null space.) 
Let $T_i$ be the natural linear maps from $\mathcal H_1$ into $\mathcal H_2$ by coordinate multipliers.
 That is, $T_i h = X_ih.$
Note $T_i$ is contractive
as $\tau(h^* (1-X_i^*X_i) h)\geq 0$ where $h$ is a principle analytic function. 

\smallskip

Let $\Lambda$ be as obtained from Remark \ref{Lambda}. Let $\mathcal H_0$ denote the constant functions.
Define $\widehat{\Lambda^{-1}}:\mathbb{C}^N \otimes \mathcal H_0 \rightarrow \mathbb{C}^N \otimes \mathcal H_1$ to be the natural image of $\Lambda^{-1}$ as a 
(block) map,
that is 
\[(\widehat{\Lambda^{-1}}(v\otimes 1))(\X) = (\Lambda(\X))^{-1} v
\]
where $N$ is the size of $\Lambda$ as a matrix.

Let $P_{\mathcal H_i \rightarrow \mathcal H_j}$ be the projection from $\mathcal H_i$ onto $\mathcal H_j$, $i>j,$ and $I_{H_j\rightarrow H_i}$ be the inclusion 
map.
Let  $Y_i = P_{\mathcal H_2 \rightarrow \mathcal H_1} T_i$. Therefore $\Y=(Y_1,\ldots, Y_d)$ is the $d$ touple of operators on $\mathcal H_1$.

Now $\Lambda(\Y)\widehat{\Lambda^{-1}}(b \otimes 1)=b \otimes 1,$ as
\begin{align*}
	\Lambda(\Y)\widehat{\Lambda^{-1}}=&\begin{bmatrix}
		\Lambda_{11}(\Y)&\Lambda_{12}(\Y)&\cdots&\Lambda_{1N}(\Y)\\
		\Lambda_{21}(\Y)&\Lambda_{22}(\Y)&\cdots&\Lambda_{2N}(\Y)\\
		\vdots&\vdots&\cdots&\vdots\\
		\Lambda_{N1}(\Y)&\Lambda_{N2}(\Y)&\cdots&\Lambda_{NN}(\Y)
	\end{bmatrix} 
	\begin{bmatrix}
		\Lambda^{-1}_{11}()&\Lambda^{-1}_{12}()&\cdots&\Lambda^{-1}_{1N}()\\
		\Lambda^{-1}_{21}()&\Lambda^{-1}_{22}()&\cdots&\Lambda^{-1}_{2N}()\\
		\vdots&\vdots&\cdots&\vdots\\
		\Lambda^{-1}_{N1}()&\Lambda^{-1}_{N2}()&\cdots&\Lambda^{-1}_{NN}()\\
	\end{bmatrix}\\
	=&\begin{bmatrix}
		\sum_{j=1}^{N}\, \Lambda_{1j}(\Y)\, \Lambda^{-1}_{j1}() &\sum_{j=1}^{N}\, \Lambda_{1j}(\Y)\, \Lambda^{-1}_{j2}()&\cdots&\sum_{j=1}^{N}\, 
\Lambda_{1j}(\Y)\, \Lambda^{-1}_{jN}()\\
		&&&\\
		\sum_{j=1}^{N}\, \Lambda_{2j}(\Y)\, \Lambda^{-1}_{j1}() &\sum_{j=1}^{N}\, \Lambda_{2j}(\Y)\, \Lambda^{-1}_{j2}()&\cdots&\sum_{j=1}^{N}\, 
\Lambda_{2j}(\Y)\, \Lambda^{-1}_{jN}()\\
		\vdots&\vdots&\cdots&\vdots\\
		\sum_{j=1}^{N}\, \Lambda_{Nj}(\Y)\, \Lambda^{-1}_{j1}() &\sum_{j=1}^{N}\, \Lambda_{Nj}(\Y)\, \Lambda^{-1}_{j2}()&\cdots&\sum_{j=1}^{N}\, 
\Lambda_{Nj}(\Y)\, \Lambda^{-1}_{jN}()\\
	\end{bmatrix}\\
	=&P_{H_2\to H_1}\begin{bmatrix}
		\sum_{j=1}^{N}\, \Lambda_{1j}(\T)\, \Lambda^{-1}_{j1}() &\sum_{j=1}^{N}\, \Lambda_{1j}(\T)\, \Lambda^{-1}_{j2}()&\cdots&\sum_{j=1}^{N}\, 
\Lambda_{1j}(\T)\, \Lambda^{-1}_{jN}()\\
		&&&\\
		\sum_{j=1}^{N}\, \Lambda_{2j}(\T)\, \Lambda^{-1}_{j1}() &\sum_{j=1}^{N}\, \Lambda_{2j}(\T)\, \Lambda^{-1}_{j2}()&\cdots&\sum_{j=1}^{N}\, 
\Lambda_{2j}(\T)\, \Lambda^{-1}_{jN}()\\
		\vdots&\vdots&\cdots&\vdots\\
		\sum_{j=1}^{N}\, \Lambda_{Nj}(\T)\, \Lambda^{-1}_{j1}() &\sum_{j=1}^{N}\, \Lambda_{Nj}(\T)\, \Lambda^{-1}_{j2}()&\cdots&\sum_{j=1}^{N}\, 
\Lambda_{Nj}(\T)\, \Lambda^{-1}_{jN}()\\
	\end{bmatrix} \\
	=&I_N \otimes P_{\mathcal H_2 \rightarrow \mathcal H_1}I_{\mathcal H_0 \rightarrow \mathcal H_2} \\
	=&I_N \otimes I_{\mathcal H_0 \rightarrow \mathcal H_1}. 
\end{align*}

In other words, 

\begin{align*}\Lambda(\Y)\widehat{\Lambda^{-1}} & = \Lambda_0 \otimes \widehat{\Lambda^{-1}} + \sum \Lambda_i \otimes Y_i \widehat{\Lambda^{-1}} \\
&=P_{\mathcal H_2 \rightarrow \mathcal H_1}[I_{\mathcal H_1 \rightarrow \mathcal H_2} \Lambda_0 \otimes \widehat{\Lambda^{-1}} + \sum \Lambda_i \otimes T_i 
\widehat{\Lambda^{-1}}  ] \\
&=I_N \otimes P_{\mathcal H_2 \rightarrow \mathcal H_1}I_{\mathcal H_0 \rightarrow \mathcal H_2} \\
&=I_N \otimes I_{\mathcal H_0 \rightarrow \mathcal H_1}.
\end{align*}

Thus, 
\begin{equation}\label{Surprise}
 \widehat{\Lambda^{-1}}(b \otimes 1)=\Lambda(\Y)^{-1}(b \otimes 1).
\end{equation}
Since $q(\Y)$ is a dim$\mathcal H_1 \times$ dim$\mathcal H_1$ matrix, we consider it as an operator on $\mathcal H_1$. In that case, its quadratic form evaluated at the constant function $1$ is 
$$\tau(1^*(b \otimes 1)^*(\Lambda(\Y)^{-1})^*M\Lambda(\Y)^{-1}(b \otimes 1) 1)= \tau((\widehat{\Lambda^{-1}}(b \otimes 1))^*M\widehat{\Lambda^{-1}}(b \otimes 1)) = \tau(q) <0,$$
where the first equality holds because of \eqref{Surprise}.
\end{proof}

We now have the sum of squares formula for the nc polydisc. 
This is different from \cite[Theorem 3.1]{Pa18} both in content and in the line of attack. 
The Montel trick for nonrational Agler type models arose previously in \cite{AgMc18}, and the Caratheodory fixed cone trick
was explicated in various places \cite{HeMc04, HeKlNe14} with the much of the state of the art being accessible in \cite{Vo18}.

\begin{proof}[{\bf Proof of Theorem \ref{Main1}}]
    Let $r<1$.
    Write $q_r(\X) = q(r\X).$
    For each $r,$ $q_r$ is in the corresponding principal cone which has dimension bounded by some constant depending on the of the number of variables and the size of the realization of the original function, independent of $r.$

    By Carath\'eodory's theorem on convex sets, the minimal number of terms in representation of $q_r$ as a sum of squares is uniformly bounded. Consequently, we may select representations with a minimal number of terms and pass to the limit as $r \to 1$. 
    
    By Montel's theorem \cite[Proposition 4.7]{AgMc15}, this limiting representation remains rational, since both the degree and the number of terms are a priori bounded.
\end{proof}

\begin{lem}\label{lem:separate var}
	Let $h$ be a hereditary nc rational function on $\mathbb{D}^d_{\rm nc}$ admitting two representations, namely
	\begin{align}\label{tworep}
		h \equiv \sum_{\text{finite}} a_i^* b_i \quad \text{and} \quad h \equiv \sum_{\text{finite}} c_i^* d_i.
	\end{align}
	Then, for any $C \in \bM_n$,
	$$
	\sum_{\text{finite}} a_i^*(\underline{Z}) C b_i(\underline{W}) = \sum_{\text{finite}} c_i^*(\underline{Z}) C d_i(\underline{W})
	$$
	for every $\underline{Z}, \underline{W} \in \mathbb{D}^d_{\rm nc}\cap \bM_n^d$.
\end{lem}

\begin{proof}
	Suppose $h$ has two representations as in \eqref{tworep}. Let $\X \in \mathbb{D}^d_{\rm nc} \cap \bM_n^d$ be a $d$-tuple of $n \times n$ matrices, and let $H$ be an $n \times n$ 
matrix satisfying $\|\X\|<H \leq I$, where $\|\X\|=\max\{\|X_i\|: 1\leq i\leq d\}$. Define $\tilde{\underline{X}} = H^{1/2} \underline{X} H^{-1/2}$. Then $\tilde{\underline{X}}\in \mathbb{D}^d_{\rm nc} \cap \bM_n^d$. Now, by \eqref{tworep}, we obtain
	$$ \sum_{\text{finite}} a_i^*(H^{1/2} {\underline{X}} H^{-1/2}) b_i(H^{1/2} {\underline{X}} H^{-1/2})
		= \sum_{\text{finite}} c_i^*(H^{1/2} {\underline{X}} H^{-1/2}) d_i(H^{1/2} {\underline{X}} H^{-1/2}) $$
which implies 
$$ \sum_{\text{finite}} H^{-1/2} a_i^*({\underline{X}}) H b_i({\underline{X}}) H^{-1/2}
		= \sum_{\text{finite}} H^{-1/2} c_i^*({\underline{X}}) H d_i({\underline{X}}) H^{-1/2} $$
and finally  \begin{align}\label{positiveH0} \sum_{\text{finite}} a_i^*({\underline{X}}) H b_i({\underline{X}})
		= \sum_{\text{finite}} c_i^*({\underline{X}}) H d_i({\underline{X}}).
	\end{align}
	Now suppose $0<H_1\leq I$ is a $n\times n$ matrix. Then $\|\X\| < (1-\|\X\|)H_1+\|\X\|I\leq I$. Therefore form \eqref{positiveH0}, we have 
	$$ \sum_{\text{finite}} a_i^*({\underline{X}}) \big((1-\|\X\|)H_1+\|\X\|I\big) b_i({\underline{X}})
		= \sum_{\text{finite}} c_i^*({\underline{X}}) \big((1-\|\X\|)H_1+\|\X\|I\big) d_i({\underline{X}})$$
and thus 
	\begin{align}\label{positiveH}\sum_{\text{finite}} a_i^*({\underline{X}}) H_1 b_i({\underline{X}})
		= \sum_{\text{finite}} c_i^*({\underline{X}}) H_1 d_i({\underline{X}})
	\end{align}
	
	Since this identity holds for all $0< H_1\leq I$, it also holds for all $H_1>0$. A standard density argument shows that \eqref{positiveH} also holds for any $H_1$. Next, substitute
	\begin{align*}
		{\underline{X}} = 
		\begin{bmatrix}
			\underline{Z} & 0 \\
			0 & \underline{W}
		\end{bmatrix}
		\quad \text{and} \quad
		H_1 =
		\begin{bmatrix}
			0 & C \\
			0 & 0
		\end{bmatrix}
	\end{align*}
	for some fixed $C$ in \eqref{positiveH}. Then we have
	\begin{align*}
		\sum_{\text{finite}}
		\begin{bmatrix}
			a_i^*(\underline{Z}) & 0 \\
			0 & a_i^*(\underline{W})
		\end{bmatrix}
		\begin{bmatrix}
			0 & C \\
			0 & 0
		\end{bmatrix}
		\begin{bmatrix}
			b_i(\underline{Z}) & 0 \\
			0 & b_i(\underline{W})
		\end{bmatrix}
		=
		\sum_{\text{finite}}
		\begin{bmatrix}
			c_i^*(\underline{Z}) & 0 \\
			0 & c_i^*(\underline{W})
		\end{bmatrix}
		\begin{bmatrix}
			0 & C \\
			0 & 0
		\end{bmatrix}
		\begin{bmatrix}
			d_i(\underline{Z}) & 0 \\
			0 & d_i(\underline{W})
		\end{bmatrix}.
	\end{align*}
	This further implies
	\begin{align}
		\begin{bmatrix}
			0 & \sum\limits_{\text{finite}} a_i^*(\underline{Z}) C b_i(\underline{W}) \\
			0 & 0
		\end{bmatrix}
		=
		\begin{bmatrix}
			0 & \sum\limits_{\text{finite}} c_i^*(\underline{Z}) C d_i(\underline{W}) \\
			0 & 0
		\end{bmatrix}.
	\end{align}
	This concludes the proof.
\end{proof}

\begin{proof}[{\bf Proof of Theorem \ref{Main2}}]
	Let $\Z = (Z_1, \ldots, Z_d)$ and $\W = (W_1, \ldots, W_d) \in \mathbb{D}^d_{\rm nc} \subseteq \bM_n^d$, and let $u, v, \alpha, \beta \in \C^n$. By Theorem \ref{Main1} and 
Lemma~\ref{lem:separate var}, there exist rational functions $r_i, s_i$, $i = 1, \ldots, N$, such that  
	\begin{align}\label{eq:1}
		\nn &u^* [\beta \alpha^* - f(\W)^* \beta \alpha^* f(\Z)] v \\
		\nn &= \sum_{i=1}^{N} u^* \big[ r_i(\W)^* (\beta \alpha^* - W_{i_j}^* \beta \alpha^* Z_{i_j}) r_i(\Z) + s_i(\W)^* \beta \alpha^* s_i(\Z) \big] v 
\end{align}
		Hence,
\begin{align} \nn & u^* \beta \alpha^* v + \sum_{i=1}^{N} u^* r_i(\W)^* W_{i_j}^* \beta \alpha^* Z_{i_j} r_i(\Z) v \\
		= &\, u^* f(\W)^* \beta \alpha^* f(\Z) v + \sum_{i=1}^{N} u^* r_i(\W)^* \beta \alpha^* r_i(\Z) v + \sum_{i=1}^{N} u^* s_i(\W)^* \beta \alpha^* s_i(\Z) 
v.
	\end{align}
Let $r(\Z) = [r_1(\Z), \ldots, r_N(\Z),s_1(\Z), \ldots, s_N(\Z)]^t,$ and let $ \widehat\Delta(\Z)$ be a $2N \times 2N$ block matrix whose only non-zero entries are the first $N$ diagonal entries and these non-zero entries come from $\{Z_1, \ldots, Z_d\}$. Using these notations, \eqref{eq:1} can be rewritten as
	\begin{align}\label{eq:norm-equality}
		u^* \beta \alpha^* v + u^* r(\W)^* \widehat\Delta(\W)^* 
		\begin{matrix}
			I_{\C^{2N}} \\ \otimes \\ \beta \alpha^*
		\end{matrix}
		\widehat\Delta(\Z) r(\Z) v 
		= \, u^* f(\W)^* \beta \alpha^* f(\Z) v
		+ u^* r(\W)^*
		\begin{matrix}
			I_{\C^{2N}} \\ \otimes \\ \beta \alpha^*
		\end{matrix}
		r(\Z) v.
	\end{align}
	Consider a linear transformation $\tilde T$ defined on the subspace 
$$\mathcal{S}:={\rm{span}} \{ 		\alpha^* v \oplus 			\big(I_{\C^{2N}}  \otimes  \alpha^*\big) 		\widehat\Delta(\Z) r(\Z)\, v: \alpha, v \in 
\C^{n} \}$$ 
by
	\begin{align*}
		\tilde T
		\begin{bmatrix}
			\alpha^* v \\[3pt]
			\begin{matrix}
				I_{\C^{2N}} \\ \otimes \\ \alpha^*
			\end{matrix}
			\widehat\Delta(\Z) r(\Z)\, v
		\end{bmatrix}
		=
		\begin{bmatrix}
			\alpha^* f(\Z) v \\[3pt]
			\begin{matrix}
				I_{\C^{2N}} \\ \otimes \\ \alpha^*
			\end{matrix}\,r(\Z)\, v
		\end{bmatrix},
	\end{align*}
	 and extend it by defining it to be $0$ on $\mathcal{S}^\perp$. Let 
\[
	\tilde T =
	\begin{bmatrix}
		\tilde{A} & \tilde{B} \\[3pt]
		\tilde{C} & \tilde{D}
	\end{bmatrix} :
	\begin{matrix}
		\C \\ \oplus \\ \C^{2N} 
	\end{matrix}
	\to
	\begin{matrix}
		\C \\ \oplus \\ \C^{2N}
	\end{matrix}
	\]
be its block matrix. Thanks to \eqref{eq:norm-equality}, $\tilde T$ is a contraction. The definition of $\tilde T$ implies
	\begin{align*}
		& \tilde{A} \alpha^* v + \tilde{B}
		\begin{matrix}
			I_{\C^{2N}} \\ \otimes \\ \alpha^*
		\end{matrix}
		\widehat\Delta(\Z) r(\Z)\, v
		= \alpha^* f(\Z)\, v, \\[3pt]
		& \tilde{C} \alpha^* v + \tilde{D}
		\begin{matrix}
			I_{\C^{2N}} \\ \otimes \\ \alpha^*
		\end{matrix}
		\widehat\Delta(\Z) r(\Z)\, v
		=
		\begin{matrix}
			I_{\C^{2N}} \\ \otimes \\ \alpha^*
		\end{matrix}\,r(\Z)\,v.
	\end{align*}
	for all $v \in \C^n$. Hence
	\begin{align}
	\label{eq:00}	& \tilde{A} \alpha^* + \tilde{B}
		\begin{matrix}
			I_{\C^{2N}} \\ \otimes \\ \alpha^*
		\end{matrix}
		\widehat\Delta(\Z) r(\Z) = \alpha^* f(\Z), \\[3pt]
		\label{eq:01}& \tilde{C} \alpha^* + \tilde{D}
		\begin{matrix}
			I_{\C^{2N}} \\ \otimes \\ \alpha^*
		\end{matrix}
		\widehat\Delta(\Z) r(\Z)
		=
		\begin{matrix}
			I_{\C^{2N}} \\ \otimes \\ \alpha^*
		\end{matrix}\,r(\Z).
	\end{align}
We may rewrite the above two equations \eqref{eq:00} and \eqref{eq:01} as follows:
	\begin{align}
	\label{eq:10}	& 	\begin{matrix}
		I_{\C} \\ \otimes \\ \alpha^*
	\end{matrix}\, \begin{matrix}
	\tilde{A} \\ \otimes \\ I_{\C^n}
	\end{matrix} + \begin{matrix}
	I_{\C} \\ \otimes \\ \alpha^*
	\end{matrix}\,
	\begin{matrix}
		\tilde{B} \\ \otimes \\ I_{\C^n}
	\end{matrix}
	\widehat\Delta(\Z) r(\Z) = \begin{matrix}
		I_{\C} \\ \otimes \\ \alpha^*
	\end{matrix}\,\begin{matrix}
	I_{\C} \\ \otimes \\ f(\Z)
	\end{matrix}, \\[3pt]
	\label{eq:11}& \begin{matrix}
		I_{\C^{2N}} \\ \otimes \\ \alpha^*
	\end{matrix}\, \begin{matrix}
		\tilde{C} \\ \otimes \\ I_{\C^n}
	\end{matrix} + \begin{matrix}
	I_{\C^{2N}} \\ \otimes \\ \alpha^*
	\end{matrix}\,
	\begin{matrix}
	\tilde{D} \\ \otimes \\ I_{\C^n}
	\end{matrix}
	\widehat\Delta(\Z) r(\Z)
	=
	\begin{matrix}
		I_{\C^{2N}} \\ \otimes \\ \alpha^*
	\end{matrix}\,r(\Z),
\end{align}
for every $\alpha\in \C^n$. Consequently, from \eqref{eq:10} and \eqref{eq:11}, we have
\begin{align*}
	& \begin{matrix}
		\tilde{A} \\ \otimes \\ I_{\C^n}
	\end{matrix} + 
	\begin{matrix}
		\tilde{B} \\ \otimes \\ I_{\C^n}
	\end{matrix}
	\widehat\Delta(\Z) r(\Z) = \begin{matrix}
		I_{\C} \\ \otimes \\ f(\Z)
	\end{matrix}, \\[3pt]
	&  \begin{matrix}
		\tilde{C} \\ \otimes \\ I_{\C^n}
	\end{matrix} + 
	\begin{matrix}
		\tilde{D} \\ \otimes \\ I_{\C^n}
	\end{matrix}
	\widehat\Delta(\Z) r(\Z)
	= \begin{matrix}
		I_{\C^{2N}} \\ \otimes \\ I_{\C^n}
	\end{matrix}\, r(\Z).
\end{align*}
Solving the above two equations we obtain
\begin{align}\label{eq:repr-0}
	 \begin{matrix}
		I_{\C} \\ \otimes \\ f(\Z)
	\end{matrix}=\begin{matrix}
	\tilde{A} \\ \otimes \\ I_{\C^n}
	\end{matrix} + 
	\begin{matrix}
	\tilde{B} \\ \otimes \\ I_{\C^n}
	\end{matrix}
	\widehat\Delta(\Z) \left(\begin{matrix}
		I_{\C^{2N}} \\ \otimes \\ I_{\C^n}
	\end{matrix} -   \begin{matrix}
	\tilde{D} \\ \otimes \\ I_{\C^n}
	\end{matrix}
	\widehat\Delta(\Z)\right)^{-1} \begin{matrix}
		\tilde{C} \\ \otimes \\ I_{\C^n}
	\end{matrix}.
\end{align}
Let $P:\C^{N}\oplus\C^{N}\to \C^{N}\oplus 0$ be the orthogonal projection given by $P(x\oplus y)=x\oplus 0$. 
Define $\Delta(\Z)$ to be the $2N \times 2N$ block diagonal matrix whose first $N$ diagonal entries are the same as those of $ \widehat\Delta(\Z)$ and the $(N+j)$th diagonal entry is the same as the $j$th one. Then \eqref{eq:repr-0} can be rewritten as follows
\begin{align}\label{eq:repr-1}
	\begin{matrix}
		I_{\C} \\ \otimes \\ f(\Z)
	\end{matrix}=\begin{matrix}
		\tilde{A} \\ \otimes \\ I_{\C^n}
	\end{matrix} + 
	\begin{matrix}
		\tilde{B}P \\ \otimes \\ I_{\C^n}
	\end{matrix}
	\Delta(\Z) \left(\begin{matrix}
		I_{\C^{2N}} \\ \otimes \\ I_{\C^n}
	\end{matrix} -   \begin{matrix}
		\tilde D P \\ \otimes \\ I_{\C^n}
	\end{matrix}
	\Delta(\Z)\right)^{-1} \begin{matrix}
		\tilde{C} \\ \otimes \\ I_{\C^n}
	\end{matrix}.
\end{align}
Now consider
\begin{align}\label{eq:matrix-T}
T = \begin{bmatrix}
	A&B\\
	C&D
\end{bmatrix}
:= \begin{bmatrix}
\tilde{A} & \tilde{B}P\\ 
\tilde{C} & \tilde{D}P
\end{bmatrix}
 = \begin{bmatrix}
	\tilde{A} & \tilde{B}\\ 
	\tilde{C} & \tilde{D}
\end{bmatrix} \begin{bmatrix}
1& 0 \\
0 & P
\end{bmatrix}.
\end{align}
Then clearly $T:\C\oplus\C^{2N}\to \C\oplus\C^{2N}$ is a contraction. Therefore, \eqref{eq:repr-1} and \eqref{eq:matrix-T} together implies 
\eqref{eq:realization-form}. This completes the proof.

\end{proof}



\begin{proof}[{\bf Proof of Corollary \ref{thm:realization-inner-rational}}]
Let $f$ be a rational function on $\mathbb{D}^d_{\rm nc}$ with the realization formula
\begin{align*}
	f(\Z) = A + B\, \Delta(\Z)\, (1 - D\, \Delta(\Z))^{-1} C
\end{align*}
given by \eqref{eq:realization-form}. Suppose the colligation matrix $T$ is unitary. We show that $f$ is inner. Since $T$ is unitary,
for each $\Z \in \mathbb{D}^d_{\rm nc}$ we have
	\begin{align}\label{eq:f*f-comput-0}
		\nn&I - f(\Z)^* f(\Z)\\
		=& I - A^* A 
		- A^* B \Delta(\Z) (I - D \Delta(\Z))^{-1} C
		- C^* (I - \Delta(\Z)^* D^*)^{-1} \Delta(\Z)^* B^* A \\
		\nn&\quad - C^* (I - \Delta(\Z)^* D^*)^{-1} 
		\Delta(\Z)^* B^* B \Delta(\Z) 
		(I - D \Delta(\Z))^{-1} C \\[6pt]
		\nn=& C^* C 
		+ C^* D \Delta(\Z) (I - D \Delta(\Z))^{-1} C
		+ C^* (I - \Delta(\Z)^* D^*)^{-1} 
		\Delta(\Z)^* D^* C \\
		\nn&\quad - C^* (I - \Delta(\Z)^* D^*)^{-1} 
		\Delta(\Z)^* (I - D^* D) 
		\Delta(\Z) (I - D \Delta(\Z))^{-1} C \\[6pt]
		\nn=& C^* (I - \Delta(\Z)^* D^*)^{-1}
		\Big[
		(I - \Delta(\Z)^* D^*)(I - D \Delta(\Z))
		+ (I - \Delta(\Z)^* D^*) D \Delta(\Z) \\
		\nn&\qquad\quad
		+ \Delta(\Z)^* D^* (I - D \Delta(\Z))
		- \Delta(\Z)^* (I - D^* D) \Delta(\Z)
		\Big]
		(I - D \Delta(\Z))^{-1} C \\[6pt]
		=& C^* (I - \Delta(\Z)^* D^*)^{-1}
		\big[ I - \Delta(\Z)^* \Delta(\Z) \big]
		(I - D \Delta(\Z))^{-1} C.
	\end{align}
	Note that on $\partial\mathbb{D}^d_{\rm nc}$, we have $I - \Delta(\Z)^* \Delta(\Z) = 0$. Let $\bH_{\bU_n^d}$ be the Haar measure on $\bU_n^d$. To conclude the proof, 
we only need to ensure that for every $n$, $(I - D\, \Delta(\Z))$ is non-singular for almost all $\Z \in \partial\mathbb{D}^d_{\rm nc} \cap \bU_n^d$ with respect to the 
Haar measure $\bH_{\bU_n^d}$. For $\theta \in [0, 2\pi]$, define
	 
	\begin{align*}
		\mathscr{A}_n(\theta)=\{e^{i\theta}\Z\in \partial\mathbb{D}^d_{\rm nc}\cap \bU_n^d:  I - D \Delta(e^{i\theta}\Z) \text{ is singular}\}.
	\end{align*}
	Then,
	\begin{align*}
		\bH_{\bU_n^d}(\mathscr{A}_n(0))=\int_{\bU_n^d} 1_{\mathscr{A}_n(0)}\, d\bH_{\bU_n^d}=\int_{\bU_n^d} 1_{\mathscr{A}_n(\theta)}\, d\bH_{\bU_n^d}.
	\end{align*}
	The last equality holds because $\bH_{\bU_n^d}$ is rotation invariant. Therefore, by Fubini's theorem, we have
	\begin{align}\label{eq:f*f-comput-1}
		\nn\int_0^{2\pi}\bH_{\bU_n^d}(\mathscr{A}_n(0))\, d\theta
		=&\int_0^{2\pi}\,\int_{\bU_n^d}\,  1_{\mathscr{A}_n(\theta)}\, d\bH_{\bU_n^d}\, d\theta\\
		=&\int_{\bU_n^d}\int_0^{2\pi} 1_{\mathscr{A}_n(\theta)}\, d\theta\,d\bH_{\bU_n^d}.
	\end{align}
	
	Note that for a fixed $\Z \in \partial\mathbb{D}^d_{\rm nc} \cap \bU_n^d$, $I - D\, \Delta(e^{i\theta}\Z)$ is singular for at most a finite number of values of 
$\theta$ as $\det(I - D\, \Delta(e^{i\theta}\Z))$ is polynomial in $\theta$ and it has only finite roots. Hence, from \eqref{eq:f*f-comput-1}, we have
	
	\begin{align*}
		\int_0^{2\pi}\bH_{\bU_n^d}(\mathscr{A}_n(0))\, d\theta= \int_{\bU_n^d}\int_0^{2\pi} 1_{\mathscr{A}_n(\theta)}\, d\theta\,d\bH_{\bU_n^d}=0.
	\end{align*}
	This shows that for every $n \in \N$, $(I - D\, \Delta(\Z))$ is non-singular almost everywhere on $\Z \in \partial\mathbb{D}^d_{\rm nc} \cap \bU_n^d$ with respect to 
the Haar measure $d\bH_{\bU_n^d}$. This proves that $f$ is inner.

Conversely, suppose $f$ is an inner nc rational function. By Theorem \ref{Main1}, there exist nc rational functions $r_i, s_i$, $1 \leq i \leq N$, for some fixed $N$, such that
\begin{align}
	f \equiv \sum_{i=1}^N r_i^* (1 - X_{i_k}^* X_{i_k}) r_i 
	+ \sum_{i=1}^N s_i^* s_i.
\end{align}
Since $f$ is inner, this forces $s_i = 0$ for all $1 \leq i \leq N$. 

Now, by repeating the arguments in the proof of Theorem \ref{Main2}, we obtain a unitary colligation matrix
\[
T = \begin{bmatrix}
	A & B \\
	C & D
\end{bmatrix}
: 
\begin{matrix}
	\C \\ \oplus \\ \C^{N}
\end{matrix}
\to
\begin{matrix}
	\C \\ \oplus \\ \C^{N}
\end{matrix},
\]
whose action on
\[
\mathcal{S} := \rm{span} \left\{ 
\alpha^* v \oplus 
\big(I_{\C^{N}} \otimes \alpha^*\big)\, \Delta(\Z)\, r(\Z)\, v 
: \alpha, v \in \C^{n} 
\right\}
\]
is given by
\begin{align*}
	T
	\begin{bmatrix}
		\alpha^* v \\[3pt]
		\big(I_{\C^{N}} \otimes \alpha^*\big)\, \Delta(\Z)\, r(\Z)\, v
	\end{bmatrix}
	=
	\begin{bmatrix}
		\alpha^* f(\Z) v \\[3pt]
		\big(I_{\C^{N}} \otimes \alpha^*\big)\, r(\Z)\, v
	\end{bmatrix},
\end{align*}
where $\Delta(\Z)$ is a block diagonal $N \times N$ matrix with the diagonal entries coming from $(Z_{1}, \dots, Z_{d})$ and $r = (r_1, \dots, r_N)^t$. Moreover, $T$ satisfies the realization \eqref{eq:realization-form}.
\end{proof}

\section{Carath\'eodory's Theorem}
Following the approach of \cite{AlBhJiKu23}, we prove a noncommutative analogue of Theorem~CR. 
In the commutative case, the result is known in the bidisc and not known yet in the polydisc.

\begin{proof}[{\bf Proof of Theorem \ref{Main}}]
Thanks to \cite[Theorem 4.1]{Ko23}, we may assume that $f$ is an nc polynomial on $\mathbb{D}^d_{\rm nc}$ satisfying $\|f(\Z)\| \le 1$.		
By Theorem~\ref{Main2}, there exist $N\in\N$ and a contraction
	 $T=\begin{bmatrix}
		A&B\\
		C&D
	\end{bmatrix}$ on $\C\oplus \C^{N}$ such that 
	\begin{align}
		f(\Z)= A + B\, \Delta(\Z)\, (1 - D\, \Delta(\Z))^{-1} C
	\end{align}
	 on $\mathbb{D}^d_{\rm nc}$ where $\Delta(\Z)$ is a block diagonal $N \times N$ matrix with the diagonal entries coming from $(Z_{1}, \dots, Z_{d})$. Let $D_{T^{*}}$ and $D_{T}$ be the defect operators 
	\begin{align}
		(I - T T^{*})^{1/2} \quad \text{and} \quad (I - T^{*} T)^{1/2},
	\end{align}
	respectively.
	 Let
	$$ D_{T^{*}} = \begin{bmatrix}
		S_{1} & S_{2} \\ S_{3} & S_{4}
	\end{bmatrix} \text{ and } D_{T} = \begin{bmatrix}
		T_{1} & T_{2} \\ T_{3} & T_{4}
	\end{bmatrix}$$
	as operators from $\C \oplus \C^{N}$ into itself. Let $\H := \C \oplus \mathbb{C}^{N}.$ Next we shall employ a finite dilation,  
i.e., for any $m\geq1,$ there is a space $\H_{m}$ consisting of the direct sum of $(m+1)$ copies of $\H$ and a unitary $U_{m}$ on it such that $T^{j} = P_{\H} 
U_{m}^{j}|_{\H}$ for $j=1,\dots,m.$ 
This idea originated with \cite{Ne85}, see also \cite{LeSh14}. 
A sequence of functions $\Theta_{m}$ induced by the 
unitaries $U_{m}$ will be the approximating sequence. To suite the idea of finite dilation to our aim, consider the space
	$$\K_m \bydef \C^{N} \oplus \H\oplus \dots \oplus \H\oplus \C \oplus \C^{N},$$
	where $\mathcal{H} $ occurs $(m-1)$ times and the block operator matrix
	$$ U_{m} :=  \begin{bmatrix}
		A & B & 0 &\dots & 0 & S_{1} & S_{2} \\
		C & D & 0 &\dots & 0 & S_{3} & S_{4} \\
		T_{1} & T_{2} & 0 & \dots & 0 & -A^{*} & -C^{*}\\
		T_{3} & T_{4} & 0 & \dots & 0 & -B^{*} & -D^{*}\\
		0 & 0 & I_{\mathcal{H}} & \dots & 0 & 0 & 0 \\
		\vdots & \vdots & \vdots & \ddots & \vdots & \vdots & \vdots \\
		0 & 0 & 0 & \dots & I_{\mathcal{H}} & 0 & 0
	\end{bmatrix} $$
	acting on the space $ \C \oplus \K_m$.

A straightforward calculation will show that $U_{m} $ is a unitary matrix. We rewrite the block matrix representation of $U_m$ as $ \begin{bmatrix}
	A & B_{m} \\ C_{m} & D_{m}
\end{bmatrix}$, where 
$$B_{m}= \begin{bmatrix} B & 0 &  \dots & 0 & S_1 & S_2
\end{bmatrix}, \;\; C_{m} = \begin{bmatrix}
	C & T_{1} & T_{3} & 0 & \dots & 0
\end{bmatrix}^{t},$$ and
$$ D_{m} = \begin{bmatrix}
	D & 0 &\dots & 0 & S_{3} & S_{4} \\
	T_{2} & 0 & \dots & 0 & -A^{*} & -C^{*}\\
	T_{4} & 0 & \dots & 0 & -B^{*} & -D^{*}\\
	0 & I_{\mathcal{H}} & \dots & 0 & 0 & 0 \\
	\vdots & \vdots & \ddots & \vdots & \vdots & \vdots \\
	0 & 0 & \dots & I_{\mathcal{H}} & 0 & 0
\end{bmatrix}.$$
Consider the following block diagonal matrices 
\begin{align*}
	\Delta_m(\Z)&=\diag(\Delta(\Z),\hat\Delta(\Z)\cdots,\hat\Delta(\Z))
\end{align*}
where $\hat\Delta(\Z)=\diag(\Delta_{11}(\Z), \Delta(\Z))$ and the matrix-valued functions $f_{m}$ given by 
\begin{align}
	f_{m}(\Z) = A + B_{m} \Delta_m(\Z) (I - D_{m} \Delta_m(\Z))^{-1} C_{m}.
\end{align}
By Theorem~\ref{thm:realization-inner-rational}, each function $f_{m}$ is rational inner, since 
\begin{align}
	\begin{bmatrix}
		A & B_{m} \\
		C_{m} & D_{m}
	\end{bmatrix}
\end{align}
is a unitary matrix. We note that 
$$ D_{m}\Delta_m(\Z) = \begin{bmatrix}
	D\Delta(\Z) & 0 &\dots & 0 & S_{3}\,\Delta_{11}(\Z) & S_{4}\,\Delta(\Z) \\
	T_{2}\,\Delta(\Z) & 0 & \dots & 0 & -A^{*}\,\Delta_{11}(\Z) & -C^{*}\Delta(\Z)\\
	T_{4}\,\Delta(\Z) & 0 & \dots & 0 & -B^{*}\,\Delta_{11}(\Z) & -D^{*}\,\Delta(\Z)\\
	0 & I_{\mathcal{H}}\hat\Delta(\Z) & \dots & 0 & 0 & 0 \\
	\vdots & \vdots & \ddots & \vdots & \vdots & \vdots \\
	0 & 0 & \dots & I_{\mathcal{H}}\hat\Delta(\Z) & 0 & 0
\end{bmatrix}.$$
Then, for a fix $k\in\N\cup\{0\}$,  a computation for matrix multiplication gives that 
\begin{align}\label{eq:main_equality}
	B_m(D_{m}\Delta_m(\Z))^k C_m = B(D\Delta(\Z))^k C, \text{ for all }m\geq k+2.
\end{align}
Fix a $\Z\in\mathbb{D}^d_{\rm nc}$. Let $\epsilon>0$. Let $k_0\in\N$ such that $\|\Delta(\Z)\|^k<\epsilon$ for all $k\geq k_0$. Then
\begin{align}\label{eq:est-0}
	\nn\|f_{m}(\Z) - f(\Z)\| & =  \|B_{m} \Delta_m(\Z) (I - D_{m} \Delta_m(\Z))^{-1} C_{m}- B \Delta(\Z) (I - D \Delta(\Z))^{-1} C\|\\
	\nn& = \left\| \sum_{k=0}^{\infty} \Big( B_m \Delta_m(\Z)\big( D_{m} \Delta_m(\Z) \big)^k C_m - B \Delta(\Z)\big( D \Delta(\Z) \big)^k C\Big)\right\|\\
	& \leq  \sum_{k=0}^{\infty} \left\|\Big( B_m\Delta_m(\Z) \big( D_{m} \Delta_m(\Z) \big)^k C_m - B\Delta(\Z) \big( D \Delta(\Z) \big)^k C\Big)\right\|
\end{align}
Note that, $\|\Delta_m(\Z)\|=\|\Delta(\Z)\|$. Therefore,
\begin{align}\label{eq:est-1}
	\nn&\sum_{k=k_0+1}^{\infty} \left\| \Big( B_m\Delta_m(\Z) \big( D_{m} \Delta_m(\Z) \big)^k C_m - B\Delta(\Z) \big( D \Delta(\Z) \big)^k C\Big)\right\|\\
	\nn\leq &~ 2\,  \| \Delta(\Z)\|\, \sum_{k=k_0+1}^{\infty} \|\Delta(\Z)\|^k\\
	 <&~ 2\epsilon\, \frac{\|\Delta(\Z)\|}{1-\|\Delta(\Z)\|}.
\end{align}
Thanks to \eqref{eq:main_equality}, for $m\geq k_0+2$,
\begin{align}\label{eq:est-2}
	 \sum_{k=0}^{k_0} \left\|\Big( B_m \big( D_{m} \Delta_m(\Z) \big)^k C_m - B \big( D \Delta(\Z) \big)^k C\Big)\right\|=0
\end{align}
Hence, thanks to \eqref{eq:est-1} and \eqref{eq:est-2}, from \eqref{eq:est-0} we have 
\begin{align*}
	\|f_{m}(\Z) - f(\Z)\|<2\epsilon\, \frac{\|\Delta(\Z)\|}{1-\|\Delta(\Z)\|} \qquad \text{ for all }  m\geq k_0+2.
\end{align*}
This shows that $f_m$ converges to $f$ point wise on $\mathbb{D}^d_{\rm nc}$, which further implies that $f_m$ converges to $f$ uniformly on every compact subset of $\mathbb{D}^d_{\rm nc}\cap \bM_n^d$ for every $n\in\N$.
\end{proof}

\noindent \textsf{Acknowledgement}: T. Bhattacharyya's work is supported by J C Bose Fellowship number JCB/2021/000041 of ANRF and the DST FIST
program-2021 [TPN-700661]. C. Pradhan acknowledges support from the Fulbright-Nehru postdoctoral fellowship.

\end{document}